\documentclass[10pt]{article}
\usepackage{amssymb}
\usepackage{amsmath}
\usepackage{amsfonts}
\usepackage{amsthm}
\usepackage{mathrsfs} 
\usepackage{graphicx}
\usepackage{epstopdf} 

\usepackage{subcaption}  

\usepackage{xargs} 

\usepackage{float} 

\usepackage{xcolor}  

\newcommand{\e}{\mathrm{e}}
\newcommand{\lam}{\lambda}
\newcommand{\ENt}[1][t]{E[N(#1)]}

\newcommandx{\mnt}[2][1=n,2=t]{m_{#1}(#2)}
\newcommandx{\mpnr}[2][1=n,2=r]{m'_{#1}(#2)}

\newcommand{\emlx}[1][x] {\e^{-\lam #1}}

\newcommandx{\sumXk}[3][1=1,2=n,3=X]{\sum_{k=#1}^{#2}{#3}_k}

\mathchardef\mhyphen="2D

\newtheorem{theorem}{\noindent Theorem}
\newtheorem{lemma}{\noindent Lemma}
\newtheorem{corollary}{\noindent Corollary}

\begin{document}

\title{Replace-after-Fixed-or-Random-Time Extensions\\of the Poisson Process}
\date{}
\author{James E. Marengo*\\Joseph G. Voelkel*\\David L. Farnsworth*\\ Kimberlee S. M. Keithley**\\
		\\
		* Rochester Institute of Technology\\
		** University of California, Santa Barbara}

\maketitle

%

\begin{abstract}
We analyze extensions of the Poisson process in which any interarrival time that exceeds a fixed value $r$ is counted as an interarrival of duration $r$. In the engineering application that initiated this work, one part is tested at a time, and $N(t)$ is the number of parts that, by time $t$, have either failed, or if they have reached age $r$ while still functioning, have been replaced. We refer to $\{N(t), t \geq 0\}$ as a replace-after-fixed-time process. We extend this idea to the case where the replacement time for the process is itself random, and refer to the resulting doubly stochastic process as a replace-after-random-time process.
\end{abstract}

\noindent%
replacement-after-random-time process; expected value of a sum; renewal process.


\section{Introduction}\label{sec:Int}
In engineering tests, parts may be tested one at a time on a test stand; in some cases, parts may be removed after a fixed amount of time $r$ so that the properties of the part at that time may be studied, or out of concern that the part may fail shortly. Here, $N(t)$ is the number of parts in $(0,t]$ that have either failed or, if they were still working at age $r$, have been replaced. We study the behavior of such a process when each part is subject to a constant hazard rate $\lam>0$. For this reason, we call $\{N(t), t \geq 0\}$ replace-after-fixed-time (RaFT) extension of the Poisson process. We also extend this process to the case where the replacement time for the process is itself random, leading to the doubly-stochastic replace-after-random-time (RaRT) process. RaRT processes might arise if, for example, different service locations replace parts but each location uses a different replacement time.

Because it spans many disciplines and includes diverse applications and theory, the literature on renewal processes is vast. See \cite{Cha1} for minimal or perfect repairs or replacements. For more general models in which the environment is changing, situations are deteriorating because of aging machinery, for example, and replacement parts may not be as good as new, see \cite{Asfaw}, \cite{Brown2}, \cite{Cha3}, \cite{Chukova}, and \cite{Yeh}. Multiple modes of failure are discussed in \cite{Wang}. The detection of changes in the arrival distribution or rate, which has an impact on all of these issues, is investigated in \cite{Brown1}. Many models are compared in  \cite{Barlow} and \cite{Cha2}. General references for these topics are \cite{Birolini}, \cite{Cha2}, and \cite{SS}.

\section{Fixed Replacement Times: the RaFT Process}
Suppose $X_1, X_2, \ldots$ are independent random variables, each having the exponential distribution with failure rate $ \lambda > 0$ and probability density function given by
\[
f(x) =
\begin{cases}
\lambda e ^{-\lambda x} & x > 0 \\
0 & x \leq 0.
\end{cases}
\]

Fix $r > 0$ and let $Y_k =$ min$(X_k, r)$ for $k \geq 1$. For $t > 0$, define $N(t) = \max\{n \geq 0: \sum\limits_{k=1}^{n} Y_k \leq t\}$. Then $\{N(t), t \geq 0 \}$ is a renewal process with interarrival times $Y_1, Y_2,\ldots$. There are two types of renewals for this process, because $N(t) = A(t) + D(t)$, where $A(t)$ is the number of components that have been replaced by time $t$ while they are still functioning (alive) and $D(t)$ is the number of components that have been replaced by time $t$ because they have failed. \\

Our purpose is to compute the joint probability distribution of $(A(t), D(t))$. This is of interest in its own right, and from this joint distribution one can derive the distribution of $N(t)$. We start with a lemma.

\begin{lemma}  \label{lemma1_kk9287}
Suppose $S_n = \sum\limits_{k=1}^{n} X_k$, $j \in \{0, 1, 2, \ldots\}$ and $jr \leq t < (j+1)r$. Then
\[
Pr(S_n \leq t, \bigcap_{k=1}^n (X_k \leq r)) = e^{-\lambda t} \sum\limits_{i=0}^{j} (-1)^i {n \choose i} \sum\limits_{k=n}^{\infty} \frac{(\lambda (t-ir))^k}{k!}.
\]
\end{lemma}

\begin{proof}
It is an easy (but lengthy) exercise to show that for $0 \leq i \leq j$, and for $k=1,\ldots,i$, that $Pr(S_n \leq t, X_k > r)$ equals
\[
\underbrace{  \int_{r}^{t-(i-1)r}   \int_{r}^{t-(i-2)r-x_1}  \cdots  \int_{r}^{t - \sum\limits_{k=1}^{i-1} x_k}}_{i\text{ fold}} \, \underbrace{ \int_{0}^{t-\sum\limits_{k=1}^{i} x_k} \cdots \int_{0}^{t-\sum\limits_{k=1}^{n-1}x_k}}_{n-i\text{ fold}} \prod_{k=1}^{n}\lambda e^{-\lambda x_k} dx_n \ldots dx_1,
\]
and that the value of this integral is

\begin{equation} \label{eqn1_kk9287}
e^{- \lambda t} \sum\limits_{k=n}^{\infty} \frac{(\lambda (t-ir))^k}{k!}.
\end{equation}
The proof now proceeds by inclusion-exclusion as follows:
\begin{eqnarray} \label{eqn2_kk9287}
& & Pr(S_n \leq t, \bigcap_{k=1}^{n} (X_k \leq r)) = Pr(S_n \leq t) - Pr (S_n \leq t, \bigcup_{k=1}^{n} (X_k > r))\nonumber\\
&=& e^{-\lambda t} \sum\limits_{k=n}^{\infty} \frac{ (\lambda t)^k }{k!} - \sum\limits_{i=1}^{j} (-1)^{i-1} \sum\limits_{1 \leq k_1 < k_2 < \cdots < k_i \leq n}  Pr(S_n \leq t, \bigcap_{l = 1}^{i} (X_{k_l} > r))\nonumber\\
&=& e ^{- \lambda t} \sum\limits_{k=n}^{\infty} \frac{(\lambda t)^k}{k!} - \sum\limits_{i = 1}^{j} (-1)^{i-1} {n \choose i} Pr(S_n \leq t,
\bigcap_{k=1}^{i} (X_k > r)).
\end{eqnarray}
Substitution of (\ref{eqn1_kk9287}) into (\ref{eqn2_kk9287}) completes the proof.
\end{proof}

\begin{theorem}[Main Theorem]\label{thm22_kk9287}
For $jr \leq t < (j + 1) r, k = 0, \ldots , j$ and $l = 0, 1, \ldots$
\[
Pr( A(t) = k, D(t) = l) = e^{-\lambda t} { k+l \choose k} \sum\limits_{i = 0}^{j-k} (-1)^i {l+1 \choose i} \frac{(\lambda (t - (i+k)r))^l}{l!}
\]
\end{theorem}

\begin{proof}
If $T_n = \sum\limits_{i=1}^{n} Y_i$, then by conditioning on which $k$ of the first $k + l$ renewals are ``alive'' we see that
\begin{eqnarray} \label{eqn3_kk9287}
& & Pr(A(t) = k, D(t) = l ) = Pr(T_{k+l} \leq t, \,T_{k+l+1} > t, \, A(t) = k)\nonumber\\
&=& { k+ l \choose k} Pr \left(T_{k+l} \leq t, \,T_{k+l+1} > t, \, \bigcap_{i=1}^l (X_i \leq r), \bigcap_{i = l + 1}^{l+k } (X_i > r)\right)\nonumber\\
& = & {k+l \choose k} e^{-\lambda k r} Pr \left(S_l \leq t - kr, \,S_l + Y_{k+l+1} > t - kr, \, \bigcap_{i=1}^l (X_i \leq r)\right).
\end{eqnarray}
Now,
\begin{eqnarray*}
&& Pr\left(S_l \leq t - kr, \, S_l + Y_{k+l+1} > t - kr, \, \bigcap_{i=1}^l (X_i \leq r)\right)\\
&=& Pr\left(S_l \leq t - kr, \, S_l + Y_{l+1} > t - kr, \, \bigcap_{i=1}^l (X_i \leq r)\right),
\end{eqnarray*}
and by conditioning on whether or not $X_{l+1} \leq r$ we see that this last probability is
\begin{eqnarray*}
&&Pr\left(S_l \leq t - kr, \, S_{l+1} > t - kr, \, \bigcap_{i=1}^{l+1} (X_i \leq r) \right) \\
&& +\; e^{-\lambda r} Pr\left(t-(k+1)r < S_l \leq t-kr, \, \bigcap_{i=1}^{l}(X_i \leq r)\right)\\
&=& (1-e^{-\lambda r}) Pr\left(S_l \leq t-kr, \, \, \bigcap_{i=1}^{l}(X_i \leq r)\right) - Pr\left(S_{l+1} \leq t-kr, \, \bigcap_{i=1}^{l+1} (X_i \leq r) \right) \\
&& +\; e^{-\lambda r} Pr\left(S_l \leq t-kr,\, \bigcap_{i=1}^{l}(X_i \leq r)\right) - e^{-\lambda r} Pr\left(S_l \leq t-(k+1)r, \, \bigcap_{i=1}^{l}(X_i \leq r)\right)\\
&=& Pr\left(S_l \leq t-kr, \, \bigcap_{i=1}^{l}(X_i \leq r)\right) - Pr\left(S_{l+1} \leq t-kr, \, \bigcap_{i=1}^{l+1}(X_i \leq r)\right)\\
&& -\; e^{-\lambda r} Pr\left(S_l \leq t-(k+1)r, \, \bigcap_{i=1}^{l}(X_i \leq r)\right).
\end{eqnarray*}
Apply Lemma \ref{lemma1_kk9287} to each of the last three probabilities to see that
\begin{eqnarray} \label{eqn3a_kk9287}
&& Pr(S_l \leq t-kr, \, S_l + Y_{k+l+1} > t-kr,\, \bigcap_{i=1}^{l}(X_i \leq r))\nonumber\\
&=& e^{-\lambda (t-kr)} \sum\limits_{i=0}^{j-k} (-1)^i {l \choose i} \sum\limits_{n=l}^{\infty} \frac{(\lambda(t-(i+k)r))^n}{n!}\nonumber\\
&& -e^{-\lambda (t-kr)} \sum\limits_{i=0}^{j-k} (-1)^i {l+1 \choose i} \sum\limits_{n=l+1}^{\infty} \frac{(\lambda(t-(i+k)r))^n}{n!}\nonumber\\
&& - e^{-\lambda (t-kr)} \sum\limits_{i=0}^{j-k-1} (-1)^i {l \choose i} \sum\limits_{n=l}^{\infty} \frac{(\lambda(t-(i+k+1)r))^n}{n!}\nonumber\\
&=&e^{-\lambda(t-kr)}
\Bigg\{ \sum\limits_{i=0}^{j-k} (-1)^i {l \choose i} \frac{(\lambda (t - (i + k)r ))^l}{l!}\nonumber\\
&&- \sum\limits_{i=1}^{j-k}(-1)^i { l \choose i-1} \sum\limits_{n=l+1}^{\infty} \frac{(\lambda (t - (i+k)r))^{n}}{n!}\nonumber\\
&&+ \sum\limits_{i = 1}^{j-k} (-1)^i {l \choose i-1} \sum\limits_{n=l}^{\infty} \frac{(\lambda (t - (i+k)r))^n}{n!} \Bigg\}\nonumber\\
&=& e^{-\lambda(t-kr)} \Bigg\{ \sum\limits_{i=0}^{j-k} (-1)^i {l \choose i} \frac{(\lambda (t - (i + k)r ))^l}{l!}
+\sum\limits_{i=1}^{j-k}(-1)^i { l \choose i-1} \frac{(\lambda (t - (i+k)r))^l}{l!} \Bigg\}\nonumber\\
&=&e^{-\lambda (t-kr)} \Bigg\{ \frac{\lambda (t-kr))^l}{l!} + \sum\limits_{i=1}^{j-k} (-1)^i {l+1 \choose i} \frac{(\lambda (t - (i+k)r))^l}{l!}\Bigg\}\nonumber\\
&=& e^{- \lambda (t-kr)} \sum\limits_{i=0}^{j-k} (-1)^i {l+1 \choose i} \frac{(\lambda (t-(i+k)r))^l}{l!}.
\end{eqnarray}

\noindent Substitution of (\ref{eqn3a_kk9287}) into (\ref{eqn3_kk9287}) completes the proof.
\end{proof}

Before proceeding further we need another lemma.

\begin{lemma} \label{lemma2_kk9287}
For $m = 1, 2,\ldots$ and $d \in \{0, 1, \ldots, m \},$
\[
\sum\limits_{k=0}^{m} (-1)^k { m \choose k} k^d =
\begin{cases}
0 &  \text{if} \quad d< m\\
(-1)^m m! & \text{if} \quad d=m.
\end{cases}
\]
\end{lemma}

\begin{proof}
Because Lemma \ref{lemma2_kk9287} follows immediately from the binomial theorem if $d=0$, we assume in what follows that $d$ is positive. Observe that for $l \in \{0, 1, \cdots, m\}$,
\begin{eqnarray} \label{eqn4_kk9287}
&&\sum\limits_{k=0}^{m} (-1)^k {m \choose k} k(k-1)\cdots(k-(l-1))\nonumber\\
&=& m(m-1)\cdots(m-(l-1)) \sum\limits_{k=l}^{m} (-1)^k \frac{(m-l)!}{(k-l)! (m-k)!}\nonumber\\
&=& m(m-1) \cdots (m-(l-1)) \sum\limits_{j=0}^{m-l} (-1)^{j+l}{m-l \choose j}\nonumber\\
&=& m(m-1) \cdots (m-(l-1))(-1)^l (-1+1)^{m-l}\nonumber\\
&=& \begin{cases}  0 & \text{if} \quad l< m\\
				(-1)^m m! & \text{if} \quad l=m,
\end{cases}
\end{eqnarray}
where the third equation follows from the binomial theorem.

For any positive integer $j$, the factorial powers $k$, $k(k-1)$,\ldots, $\displaystyle\prod_{i=1}^{j} (k-(i-1))$ are a basis for the vector space of real polynomials of degree at most $j$, so there are unique real constants $c_{1j}, c_{2j}, \ldots, c_{jj}$ (with $c_{jj}$ = 1) such that
\[
k^j = \sum\limits_{l=1}^{j} c_{lj} k(k-1) \cdots (k-(l-1))
\]
Hence
\begin{eqnarray*}
\sum\limits_{k=0}^{m} (-1)^k {m \choose k} k^d &=&\sum\limits_{k=0}^{m}(-1)^k {m \choose k} \sum\limits_{l=1}^{d} c_{ld} k(k-1)\cdots(k-(l-1)) \\
&=& \sum\limits_{l=1}^{d} c_{ld} \sum\limits_{k=0}^{m} (-1)^k {m \choose k} k (k-1) \cdots (k-(l-1)),
\end{eqnarray*}
and it follows from \ref{eqn4_kk9287} that
\begin{eqnarray*}
\sum\limits_{k=0}^{m} (-1)^k {m \choose k} k^d &=& c_{dd} \sum\limits_{k=0}^{m} (-1)^k {m \choose k} k(k-1) \cdots (k-(d-1))\\
&=& \begin{cases}
  0 \quad & \text{if} \quad d< m \\
  (-1)^m m! & \text{if} \quad  d=m.
\end{cases}
\end{eqnarray*}
\end{proof}
We close this section with three corollaries.

\begin{corollary} \label{corollary1_kk9287}
For all $t > 0$, $D(t) \sim Poisson  (\mu = \lambda t)$.
\end{corollary}

\begin{proof}
Let $j= \left \lfloor{t/r}\right \rfloor $ so that $jr \leq t < (j+1)r$. For $l = 0, 1,\ldots$
\[Pr (D(t) = l) = \sum\limits_{k=0}^{j} Pr(A(t) = k, D(t) = l )
\]
\[
= e^{- \lambda t} \sum\limits_{k=0}^{j} \sum\limits_{i=0}^{j-k} (-1)^i { l+1 \choose i} {k+l \choose l} \frac{(\lambda (t - (i+k)r))^l}{l!}.
\]
Letting $m=i+k$, we see that this sum is
\[
e^{-\lambda t} \sum\limits_{m=0}^{j} \sum\limits_{i=0}^{m}(-1)^i {l+1 \choose i} {l+m-i \choose l} \frac{(\lambda (t-mr))^l}{l!}
\]
\[
= \underbrace{e^{- \lambda t} \frac{(\lambda t)^l}{l!}}_\text{m = 0} + \sum\limits_{m=1}^{j} e^{- \lambda t} \frac{(\lambda (t-mr))^l}{l!} \sum\limits_{i=0}^{m} (-1)^i {l+1 \choose i} {l+m-i \choose l}
\]
Because ${ l+m - i \choose l}$ is a polynomial in $i$ of degree $l$, it follows from Lemma \ref{lemma2_kk9287} that
\[
\sum\limits_{i=0}^{m} (-1)^i {l+1 \choose i} {l+m-i \choose l} = 0.
\]
Hence $P(D(t) = l ) = e^{- \lambda t} \frac{(\lambda t)^l}{l!}$ and so $D \sim Poisson(\mu = \lambda t)$.
\end{proof}

Note that this result is not surprising in light of the fact that the exponential distribution is memoryless. Given that an exponential random variable with failure rate $\lambda$ has ``survived'' for an amount of time which is a multiple of $r$, the conditional distribution of its remaining life is still exponential with failure rate $\lambda$. Hence a used component is as good as a new one and the act of replacing a functioning component does not change the distribution of time until the next failure. That is, the times at which failed components are replaced form an ordinary Poisson process with rate $\lambda$, even though one or more replacements of still functioning components may have taken place between any pair of consecutive failures or before the first failure.

The distribution of $A(t)$ is in general quite complicated. However, our next corollary gives its expectation.

\begin{corollary} \label{corollary2_kk9287}
For $jr \leq t < (j+1)r$,
\[
 E [A(t)] = \sum\limits_{m=1}^{j} e^{- \lambda mr} (1+\lambda (t-mr))
 \]
\end{corollary}

\begin{proof}
From Theorem \ref{thm22_kk9287} it follows that
$$E [A(t)] = \sum\limits_{l=0}^{\infty} \sum\limits_{k=0}^{j} k Pr(A(t) = k, D(t) = l)$$
$$= e^{- \lambda t} \sum\limits_{k=0}^{j} \sum\limits_{i=0}^{j-k} \sum\limits_{l=0}^{\infty} (-1)^i k {k+l \choose k} {l+1 \choose i} \frac{(\lambda (t - (i+k)r))^l}{l!}.$$
Now let $m = i+k$ to conclude that
  \begin{multline} \label{eqn5_kk9287}
 E[A(t)]= e^{- \lambda t} \sum\limits_{m=0}^{j} \sum\limits_{l=0}^{\infty} \frac{(\lambda (t-mr))^l}{l!} \sum\limits_{i=0}^{m} (-1)^i {l+1 \choose i} (m-i) {m-i+l \choose l}.
\end{multline}
Because $(m-i) {m-i+l \choose l}$ is a polynomial in $i$ of degree $l+1$ with leading coefficient $\frac{(-1)^{l+1}}{l!}$, it follows from Lemma \ref{lemma2_kk9287} that
\[
\sum\limits_{i=0}^m (-1)^i {l+1 \choose i} (m-i) {m-i+l \choose l} = \frac{(-1)^{l+1}}{l!} (-1)^{l+1} (l+1)! = 1+ l.
\]
Substitution into (\ref{eqn5_kk9287}) yields
\[
 E [A(t)] = e^{-\lambda t} \sum\limits_{m=1}^{j} \sum\limits_{l=0}^{\infty} (1+l) \frac{(\lambda (t-mr))^l}{l!}
 \]
 \[
 = \sum\limits_{m=1}^j e^{-\lambda mr} \sum\limits_{l=0}^{\infty} (1+l) e^{-\lambda (t-mr)} \frac{(\lambda (t-mr))^l}{l!}
 \]
 \[= \sum\limits_{m=1}^{j} e^{- \lambda mr} (1+\lambda (t-mr)).
 \]

\end{proof}

Next we consider the behavior of $\ENt/t$.
\begin{corollary}\label{cor:Ntr2r2}

\begin{enumerate}
  \item If $r < 1/\lambda$, then $\ENt/t$ is strictly increasing for $t \in [r, \infty)$.
  \item For every $r$ and $\lambda$, there is a positive integer $N$ such that $\ENt/t$ is strictly increasing on $[Nr, \infty)$
\end{enumerate}
 \end{corollary}
\begin{proof}
For statement 1, it follows from Corollaries \ref{corollary1_kk9287} and \ref{corollary2_kk9287} that for $n \in \{ 1, 2, \ldots\}$ and $t \in [nr, (n+1)r)$ that
\[ \ENt/t = \lambda (1 + u_n) + \frac{u_n-\lambda r v_n}{t},\]
where
\[ u_n = \sum\limits_{k=1}^{n} e^{-\lambda k r} = \frac{e^{-\lambda r}}{1- e^{-\lambda r}} (1 - e^{-\lambda n r})\]
and
\[ v_n = \sum\limits_{k=1}^{n} ke^{-\lambda k r} = \frac{e^{-\lambda r}}{(1-e^{-\lambda r})^2} \left( 1 - (n+1)e^{-\lambda n r} + ne^{-\lambda (n+1)r)}\right).\]

\noindent If $r > 1/\lambda$, then because $u_1 = v_1$ and $u_n < v_n$ for $n \geq 2$, we see that $u_n < \lambda r v_n$ for $n \geq 1$ and conclude that $\ENt/t$ is strictly increasing for $t \in [r, \infty)$.

Next, fix $r$ and $\lambda$ and suppose that $r \leq 1/\lambda$. To prove statement 2 it will suffice to show that for sufficiently large $n$, $u_n < \lambda r v_n$, or equivalently, $\lambda r v_n/u_n > 1$. But $$\frac{\lambda r v_n}{u_n} \rightarrow \frac{\lambda r}{1-e^{-\lambda r}} > 1$$ as $n \rightarrow \infty$, so statement 2 follows.
\end{proof}

It should be noted that, because the $\left\{u_n\right\}_{n=1}^\infty$ and $\left\{v_n\right\}_{n=1}^\infty$ sequences are both convergent and because $1+u_n\rightarrow 1/(1-\emlx[r])$ as $n\rightarrow\infty$, it follows that $\ENt /t \rightarrow\lam / (1-\emlx[r])$ as $t\rightarrow\infty$, as guaranteed by the Elementary Renewal Theorem (see \cite{Barlow} or \cite{Ross}).

Examples of Corollary~\ref{cor:Ntr2r2} and the limiting behavior of $\ENt /t$ are shown in Figure~\ref{fig:ENtt} where in both figures the jump points of $t$, that is $t=r, 2r, 3r, \ldots$, are indicated by $\times$'s. In (a), $\ENt /t$ is increasing for $t \geq r$ as guaranteed by the first statement in the Corollary; in (b), $\ENt /t$ is at first decreasing for $t>r$ (except at the jumps), but eventually is increasing, as guaranteed by the second statement.

\begin{figure}[H]
\centering
\begin{subfigure}{.5\textwidth}
  \centering
  \includegraphics[width=.9\linewidth]{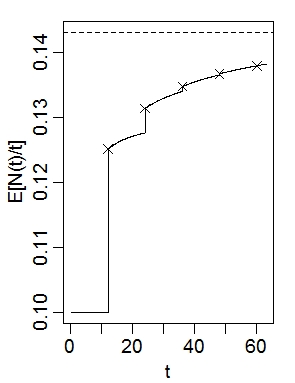}
  \caption{$r=12 > 1/\lam$}
  \label{fig:sub1}
\end{subfigure}%
\begin{subfigure}{.5\textwidth}
  \centering
  \includegraphics[width=.9\linewidth]{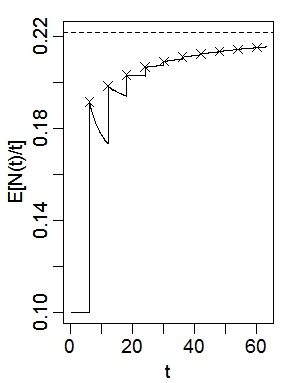}
  \caption{$r=6 < 1/\lam$}
  \label{fig:sub2}
\end{subfigure}
\caption{$E[N(t)]/t$ vs. $t$, for $\lam=1/10$. Dashed line is $\lim_{t\rightarrow\infty} E[N(t)]/t$.}
\label{fig:ENtt}
\end{figure}

\section{Random Replacement Times: the RaRT Process}\label{sec:RRT}
Next we consider the situation in which the replacement time $r$ is a realization of an associated random variable $R$, resulting in a doubly-stochastic process that we denote as a replace-after-random-time (RaRT) process. In what follows, we assume that $R$ has an absolutely continuous distribution on $(0, \infty)$ with probability density function (pdf) $f(r)$, and that for $t > 0$, $N(t) = \max\{n \geq 0: \sum\limits_{k=1}^{n}$min$(X_k, R) \leq t \}$, where, as before, $X_1, \, X_2,\, \ldots$ are independent random variables, each having the exponential distribution with failure rate $\lambda$. Finally, we assume that $R$ is independent of $X_1, \, X_2, \,\ldots$.

The following theorem expresses the expectation of $N(t)$ in terms of the pdf of the replacement time.

\begin{theorem}
\label{thm:EENtR}
\begin{equation}
\label{eqEENtR:soln}
E[N(t)] = \sum\limits_{m=1}^{\infty} \int_{0}^{t/m} e^{-\lambda mr} \left( 1 + \lambda (t-mr)\right) f(r) \,dr. \end{equation}
\end{theorem}
\begin{proof}
By conditioning on the value of $R$ and using the fact that $R$ is independent of $X_1, \, X_2, \, \ldots$, we see from Corollaries \ref{corollary1_kk9287} and \ref{corollary2_kk9287} that
\begin{eqnarray*}
E[N(t)] &=& \int_{0}^{\infty} E \left( N(t) \mid R = r \right) f(r) \,dr \\
&=& \int_{0}^{\infty} \left( \lambda t + \sum\limits_{m-1}^{ \lfloor{\frac{t}{r}}\rfloor } e^{-\lambda m r} (1 + \lambda(t-mr)) \right) f(r)\,dr\\
&=& \lambda t + \sum\limits_{n=1}^{\infty} \int_{t/(n+1)}^{t/n} \sum\limits_{m=1}^{n} e^{-\lambda m r} ( 1 + \lambda(t-mr)) f(r) \,dr\\
&=& \lambda t + \sum\limits_{m=1}^{\infty} \sum\limits_{n=m}^{\infty} \int_{t/(n+1)}^{t/n} e^{-\lambda mr} (1 + \lambda (t-mr)) f(r) \,dr\\
&=& \lambda t + \sum\limits_{m=1}^{\infty} \int_{0}^{t/m} e^{-\lambda mr} ( 1 + \lambda(t-mr)) f(r) \,dr,
\end{eqnarray*}
where the fourth equation follows by an application of Fubini's Theorem.
\end{proof}

In the next two corollaries, we consider the behavior of $E[N(t)]$ based on the behavior of $f$ near 0.
\begin{corollary}[of Theorem~\ref{thm:EENtR}]\label{cor:EENtR1}
 If for some $\epsilon > 0$ and $\delta > 0$, $f(r)>\epsilon$ for $0<r<\delta$, then $E[N(t)]=\infty$.
\end{corollary}
\begin{proof}
    For $k$ sufficiently large, $t/k < \delta$, and hence
    \begin{equation*}
    \int_0^{t/k} e^{- \lambda kr}\left(1 + \lambda (t -kr) \right) f(r) \, \,dr > \epsilon \int_0^{t/k} e^{- \lambda kr}\left(1 + \lambda (t -kr) \right) \, \,dr = \epsilon \, t/k.
    \end{equation*}
    It follows from Theorem \ref{thm:EENtR} that $E[N(t)] = \infty$.
\end{proof}

\begin{corollary}[of Theorem~\ref{thm:EENtR}]\label{cor:EENtR2}
 If for some $\epsilon > 0$ and $\delta > 0$, $f(r) < r^\epsilon$ for $0<r<\delta$, then $E[N(t)] < \infty$.
\end{corollary}
\begin{proof}
    For $k$ sufficiently large, $t/k < \delta$, and hence
    \begin{eqnarray*}
    &   & \int_0^{t/k} e^{- \lambda kr}\left(1 + \lambda (t -kr) \right) f(r) \, dr \\
    & < & (1 + \lam t) \int_0^{t/k} \emlx[kr] f(r) \, dr \\
    & < & (1 + \lam t) \int_0^{t/k} \emlx[kr] r^\epsilon \, dr \\
    & = & \frac{1 + \lam t}{(\lam k)^{1+\epsilon}} \int_0^{\lam t} e^{-u} u^\epsilon \, du.
    \end{eqnarray*}
    It follows from Theorem \ref{thm:EENtR} that $E[N(t)] < \infty$.
\end{proof}

We next consider $E[N(t)]$ in (\ref{eqEENtR:soln}) for some particular families of distributions of $R$. We restrict ourselves to two families of distributions of $R$ for which closed-form solutions exist: the exponential distribution with hazard rate $\nu > 0$ and location shift $\eta \geq 0$, and the uniform distribution on $(a,b)$ for $a \geq 0$.

For the exponential replacement case, the integrand in (\ref{eqEENtR:soln}) is, for $r > \eta$,
\[
\e^{-k\lam r}\left(\lam\left(t-kr\right)+1\right)\nu\e^{-\nu\left(r-\eta\right)}, \nonumber
\]

\noindent whose indefinite integral is
\[
\dfrac{\e^{-k\lam r}\nu\left(\lam\left(\nu+k\lam\right)\left(kr-t\right)-\nu\right)\e^{-\nu(r-\eta)}}{\left(\nu+k\lam\right)^2}. \nonumber
\]

For the uniform replacement case, for $r \in (a,b)$, the integrand is
\[
\dfrac{\e^{-k\lam r}\left(\lam\left(t-kr\right)+1\right)}{b-a}, \nonumber
\]
\noindent whose indefinite integral is
\[
\dfrac{\e^{-k\lam r}\left(kr-t\right)}{\left(b-a\right)k}.
\]

From Corollary~\ref{cor:EENtR1}, we see that for a replacement distribution that is either exponential with $\eta=0$, or uniform with $a=0$, that $E[N(t)]=\infty$.

To illustrate the effects of a random replacement time, we consider $\lam=1/10$ and five uniform distributions for $R$. We let $a=6^\text{--}$ and $b=(6^\text{+},7,7.5,11,40)$. See Figure~\ref{fig:ENttUab}, and note that $R\sim U(6^\text{--},6^\text{+})$ is equivalent to the RaFT process shown in Figure~\ref{fig:ENtt}(b). The open symbols are the values of $\ENt/t$ at $t=(6,12,\ldots,36)$; the solid symbols correspond to the $b$-value increments; for example, for $b=7$, they appear at $t=(7,14,21,28,35)$. The values from 12 to 14, for example, indicate $t$ regions where two consecutive replacements are possible.

\begin{figure}[H] 
\centering
  \scalebox{0.4}{\includegraphics{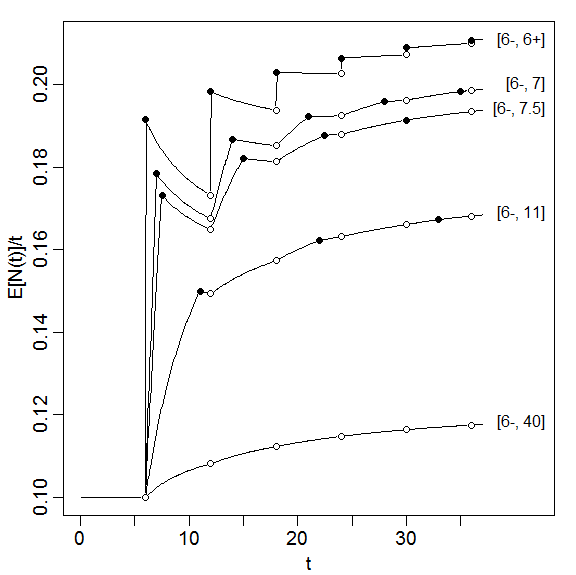}}
  \caption{$\ENt/t$ vs. $t$ for $\lam=1/10$, for each of five $U(6^\text{--},b)$ $R$ distributions.}
  \label{fig:ENttUab}
\end{figure}

{\bf Acknowledgement} The authors would like to acknowledge Mr. Robert Moses, a Reliability Engineer at General Motors, for motivating this research.


\begin{thebibliography}{99}
\footnotesize

\bibitem{Asfaw}
{\sc Asfaw, Z.~G. and Lindqvist, B.} (2015). Extending minimal repair models for repairable systems: a comparison of dynamic and heterogeneous extensions of a nonhomogeneous Poisson process. {\em Reliab. Eng. Syst. Safe.} {\bf 140,} 53--58.

\bibitem{Barlow}
{\sc Barlow, R.~E. and Proschan, F.} (1975).  {\em Statistical Theory of Reliability and Life Testing: Probability Models}. Holt, Rinehart and Winston, New York.

\bibitem{Birolini}
{\sc Birolini, A.} (2017). {\em Reliability Engineering: Theory and Practice}, 8th~edn. Springer, Berlin, Germany.

\bibitem{Brown1}
{\sc Brown, M. and Proschan, F.} (1983). Imperfect repair. {\em J. Appl. Prob.} {\bf 20,} 851--862.

\bibitem{Brown2}
{\sc Brown, M.} (2008). Monitoring a Poisson process in several categories subject to changes in the arrival rates. {\em Stat. Prob. Lett.} {\bf 78,} 2637--2643.

\bibitem{Cha1}
{\sc Cha, J.~H. and Finkelstein, M.} (2018). On preventive maintenance under different assumptions on the failure/repair processes. {\em Qual. Reliab. Engng. Int.} {\bf 34,} 66-77.

\bibitem{Cha2}
{\sc Cha, J.~H. and Finkelstein M.} (2018). {\em Point Processes for Reliability Analysis: Shocks and Repairable Systems}. Springer, Cham, Switzerland.

\bibitem{Cha3}
{\sc Cha, J.~H. and Mi, J.} (2007). Study of a stochastic failure model in a random environment. {\em J. Appl. Prob.} {\bf 44,} 151--163.

\bibitem{Chukova}
{\sc Chukova, S., Dimitrov, B., and Garrido, J.} (1993). Renewal and nonhomogeneous Poisson processes generated by distributions with periodic failure rate. {\em Stat. Prob. Lett.} {\bf 17,} 19--25.

\bibitem{Ross}
{\sc Ross, S.~M.} (2014). {\em Introduction to Probability Models}, 11th~edn. Academic Press, Boston.

\bibitem{SS}
{\sc S\'{a}nchez-Silva, M. and Klutke, G.~A.} (2016). {\em Reliability and Life-Cycle Analysis of Deteriorating Systems}, Springer, Cham, Switzerland.

\bibitem{Wang}
{\sc Wang, G.~J. and Zhang Y.~L.} (2013). Optimal repair-replacement policies for a system with two types of failures. {\em Eur. J. Oper. Res.} {\bf 226,} 500--506.

\bibitem{Yeh}
{\sc Yeh, L.} (1988). A note on the optimal replacement problem. {\em Adv. Appl. Prob.} {\bf 20,} 479--482.

\end{thebibliography}
\end{document}